\documentclass[12pt,a4paper]{article}
\usepackage{amsmath,latexsym,amsfonts,amssymb,mathrsfs,ifpdf}
\usepackage{graphicx,color}
\usepackage{hyperref,amsthm}

\newtheorem{theorem}{Theorem}[section]

\def\N{\mathbb{N}}
\newcommand{\calO}{\mathcal{O}}
\newcommand{\calB}{\mathcal{B}}

\newcommand{\calG}{{\mathcal  G}}
\newcommand{\calL}{\mathcal{L}}
\newcommand{\calN}{\mathcal{N}}

\newcommand{\bpi}{\mathbf{\pi}}

\newcommand{\dd}{{\rm d}}

\newcommand{\vf}{{\bf f}}

\newcommand{\bg}{{\bf g}}

\newcommand{\vu}{{\bf u}}
\newcommand{\vv}{{\bf v}}

\newcommand{\rd}{{\rm d}}
\newcommand{\bchi}{\boldsymbol \chi}

\newcommand{\balpha}{\boldsymbol \alpha}

\newcommand{\vV}{{\bf V}}
\newcommand{\vw}{{\bf w}}

\newcommand{\bsx}{{\boldsymbol{x}}}
\newcommand{\bsy}{{\boldsymbol{y}}}

\newcommand{\vsigma}{\boldsymbol{\sigma}}

\newcommand{\R}{\mathbb{R}}

\newcommand{\calS}{\mathcal{S}}

\numberwithin{equation}{section}

\newcommand{\veps}{{\boldsymbol{\varepsilon}}}

\numberwithin{theorem}{section}
\newcommand{\TheTitle}{Bayesian inference  calibration of the modulus of elasticity}
\newcommand{\TheAuthors}{Dick, Le Gia, Mustapha}
\title{{\TheTitle}\thanks{This work was supported by the Australian Research Council
grant DP220101811.}
}

\author{J. Dick, Q. T. Le Gia, K. Mustapha \thanks{School of Mathematics and Statistics, University of New South Wales, Sydney, Australia}
}

\ifpdf
\hypersetup{
  pdftitle={\TheTitle},
  pdfauthor={\TheAuthors}
}
\fi

\begin{document}

\maketitle
\begin{abstract}
This work uses the Bayesian inference technique to infer the Young modulus from the stochastic linear elasticity equation. The Young modulus is modeled by a finite Karhunen Lo\'{e}ve expansion, while the solution to the linear elasticity equation is approximated by the finite element method. The high dimensional integral involving the posterior density and the quantity of interest is approximated by a higher-order quasi-Monte Carlo method.
\end{abstract}

\section{Introduction}

The modulus of elasticity is a fundamental material parameter that characterizes stiffness and is critical in structural and mechanical design. However, its estimation is often subject to uncertainty due to measurement noise, material variability, and limited data availability. Traditional methods typically provide point estimates without accounting for these uncertainties, which can lead to overconfident predictions and potential design risks.

Bayesian inference~\cite{Alexanderian_2021,BECK2018523,Dashti-Stuart2017,Sarah-Eberle-Blick-2024} offers a probabilistic framework for estimating the modulus of elasticity by combining observed data with prior knowledge. Instead of yielding a single value, it produces a full posterior distribution that captures the uncertainty in the estimate. This allows for more informed decision-making, particularly in applications involving heterogeneous materials, small sample sizes, or complex testing conditions. Moreover, Bayesian methods are inherently flexible, allowing for sequential updating as new data becomes available. This makes Bayesian inference a powerful and transparent tool for improving the reliability and interpretability of elasticity measurements in engineering and materials science.

The equation governing small elastic deformations of a body $\Omega$ in~$\R^2$ with polygonal boundary can be written as: the displacement vector field $\vu(\cdot, \bsy)$ satisfies  
\begin{equation}
 -\nabla \cdot \vsigma(\bsy;\vu(\bsx,\bsy)) = \vf(\bsx) \quad \text{for } \bsx \in \Omega, \label{eq:L1}
\end{equation}
subject to homogeneous boundary conditions,  $\vu(\cdot,\bsy) = {\bf 0}$ on $\partial \Omega$ (the boundary of $\Omega$), with $\bsy$ being parameter vectors describing randomness. The parametric Cauchy stress tensor $\vsigma(\bsy;\cdot) \in [L^2(\Omega)]^{2\times 2}$ is defined as 
\[\vsigma(\bsy;\vu(\cdot,\bsy)) = \lambda(\cdot,\bsy)\Big(\nabla\cdot \vu(\cdot,\bsy)\Big) I 
+ 2\mu(\cdot,\bsy) \veps(\vu(\cdot,\bsy))\quad{\rm on}\quad \Omega, \] 
where  the symmetric strain tensor
 $ \veps(\vu) := \frac{1}{2} (\nabla \vu + (\nabla \vu)^T).$
Here, $\vf$ is the body force per unit volume, and $I$ is the identity tensor.   The gradient ($\nabla$) and the divergence ($\nabla \cdot$) are understood to be with respect to the physical variable $\bsx \in \Omega$.  The inhomogeneous random parameters $\mu$ and $\lambda$ are represented  in terms of the  Young’s modulus $E$ as: 
\begin{equation}\label{eq:mulambda}
    \mu(\bsx,\bsy)=\frac{E(\bsx,\bsy)}{2(1+\nu)}\quad{\rm and}\quad \lambda(\bsx,\bsy)=\frac{E(\bsx,\bsy)\,\nu}{(1+\nu)(1-2\nu)},
\end{equation}
where the constant  $0<\nu<1/2$  is the Poisson ratio of the elastic material and $E$ is expressed in the following Karhunen Lo\'{e}ve expansions: 
\begin{equation}\label{KLexpansion}
 E(\bsx,\bsy) = E_0(\bsx) +\sum_{j=1}^\infty y_j \psi_j(\bsx),
\end{equation}
where $\bsy= (y_j)_{j\ge1}$ belongs to $U:=(-\frac{1}{2},\frac{1}{2})^\N$ consisting of a countable number of parameters $y_j$,  which are assumed to be i.i.d. uniformly distributed. The ${\psi_j}$ are $L^2(\Omega)$-orthogonal basis functions, which are assumed to belong to $L^\infty(\Omega)$ in the convergence analysis section. 

In this work, we will use the Bayesian inference technique to determine the conditional distribution of the solution to \eqref{eq:L1} given some data measurements. We consider the approximation of the conditional density with some quantity of interest applied to the solution using higher-order quasi-Monte Carlo (QMC) method and finite element methods.


\section{The forward problem}\label{sec:forward}

We introduce  the following vector function spaces and the associated norms. Let   $\vV:=[H^1_0(\Omega)]^2$, and the associated norm be   $\|\cdot \|_{\vV}.$  For $J\ge 1$, the norm on the vector Sobolev space ${{\bf H}^J}:=[H^J(\Omega)]^2$ is denoted by $\|\cdot\|_{{{\bf H}^J}}$. For the ${\bf L}^2(\Omega)$-norm, we use the notation $\|\cdot\|.$  Finally,  $\vV^*$ is denoted  the dual space of $\vV$ with respect to the ${\bf L}^2(\Omega)$  inner product, with norm denoted by $\|\cdot\|_{\vV^\ast}$.

The weak formulation for the forward problem is described as follows.
Given $\bsy \in U$, find $\vu(\cdot,\bsy) \in \vV$ satisfying
\begin{equation}\label{para weak}
 \calB(\bsy;\vu, \vv) = \ell(\vv), \quad \text{for all} \quad \vv \in \vV,
\end{equation}
where  the bilinear form $\calB(\bsy;\cdot,\cdot)$ and the linear functional $\ell(\cdot)$ are defined by
\begin{equation*}\label{eq: bilinear}
\calB(\bsy;\vu, \vv) := \int_\Omega [2\mu(\bsy)\, \veps(\vu):\veps(\vv)+\lambda(\bsy) \nabla \cdot \vu \nabla \cdot \vv] \,d\bsx~~{\rm and}~~
\ell(\vv) :=\int_\Omega \vf \cdot \vv \,d\bsx.
\end{equation*}
The colon operator is the inner product between tensors. For the well-posedness of the  elastic problem \eqref{eq:L1}, we assume throughout the paper that
$E_{\min} \le E\le E_{\max}  \text{ on } \Omega \times U,$
for some positive constants $E_{\min}$ and  $E_{\max}$. 
This assumption leads to 
\begin{equation*}
\frac{ E_{\min} }{2(1+\nu)} \le \mu \le \frac{ E_{\max}}{2(1+\nu)}~~{\rm and}~~\frac{\nu\, E_{\min}}{(1+\nu)(1-2\nu)} \le \lambda \le\frac{ \nu\,E_{\max}}{(1+\nu)(1-2\nu)}
\end{equation*}
on $\Omega \times U.$ Then, for every $f \in \vV^\ast$ and $\bsy \in U$, the parametric weak formulation problem \eqref{para weak} has a unique solution \cite[Theorem 2]{ClarkeEtAl2024}. 
For a practical implementation,  we truncate the Karhunen Lo\'{e}ve expansion of $E$ in \eqref{KLexpansion} by $E_s = E_0 + \sum_{j=1}^s y_j \psi_j,$
that is,  assuming that $y_j=0$ for $j>s.$ Then,  with $\bsy_s=(y_1,y_2,\cdots,y_s,0,0,\cdots),$ and by using \eqref{eq:mulambda}, the errors from truncating the Lam\'e parameters are 
\begin{gather*}
    \|\mu(\cdot,\bsy)-\mu(\cdot,\bsy_s)\|_{L^\infty(\Omega)}\le \frac{1}{2(1+\nu)}\sum_{j \ge s+1} \|\psi_j\|_{L^\infty(\Omega)}~~{\rm and}\\
 \|\lambda(\cdot,\bsy)-\lambda(\cdot,\bsy_s)\|_{L^\infty(\Omega)}\le \frac{1}{1-2\nu}\sum_{j \ge s+1} \|\psi_j\|_{L^\infty(\Omega)}\,.
\end{gather*}
For the convergence of the series truncation in \eqref{KLexpansion}, we order the functions $\|\psi_j\|_{L^\infty(\Omega)}$ in a decreasing order, that is, $\| \psi_j \|_{L^\infty(\Omega)} \ge
\| \psi_{j+1} \|_{L^\infty(\Omega)}$ for $j\ge 1.$ In addition,     we assume that   for some $0<p<1$
  \begin{equation}\label{ass A11}
 \sum_{j=1}^\infty \|\psi_j\|^p_{L^\infty(\Omega)} < \infty~~{\rm or}~~
 \sum_{j=s+1}^\infty \|\psi_j\|_{L^\infty(\Omega)} \le  C s^{1-1/p}.
\tag{A2}
\end{equation}
For the Galerkin finite element convergence analysis,  assume that
\begin{equation}\label{ass A2}
\begin{aligned} 
&E_0  \in W^{\theta,\infty}(\Omega),\quad
\sum_{j=1}^\infty \|\psi_j\|_{W^{\theta,\infty}(\Omega)} < \infty.
\end{aligned}           
\tag{A3}
\end{equation}

The weak formulation for the truncated problem is: for every  $\bsy_s \in U$, find $\tilde \vu(\cdot,\bsy_s) \in \vV$ satisfying
\begin{equation}\label{para weak truncated}
 \calB(\bsy_s;\tilde \vu, \vv) = \ell(\vv), \quad \text{for all} \quad \vv \in \vV.
\end{equation}
%
Following the proof of  \cite[Theorem 5]{ClarkeEtAl2024}, 
we obtain error results in the next theorem.  The   (generic) constant $C$  depends on  $\Omega$, $\nu,$  $E_{\min}$ and $E_{\max}$.
\begin{theorem}\label{convergence calL u-us}
Assume that \eqref{ass A11} is satisfied for some $0<p<1$, and  $\chi:\vV \to \R$ is a bounded  linear functional, ($|\chi(\vw)|\le \|\chi\|_{\vV^*}\|\vw\|_\vV$ for all $\vw \in \vV$). Then, for every $\vf \in \vV^*$, $\bsy\in U$, and $s \in \N$, the solution $\tilde \vu$ of the truncated parametric weak problem
\eqref{para weak truncated}
satisfies
\begin{equation*}
|\chi(\vu(\cdot, \bsy))
-\chi(\tilde \vu(\cdot, \bsy_{s}))| 
\le C\, s^{1-1/p}\,   \|\vf\|_{\vV^*} \|\chi\|_{\vV^*}~~{\rm for~every}~\bsy \in U\,.   
\end{equation*}
\end{theorem}

Next,  we approximate the truncated solution $\tilde \vu$ using the Galerkin finite element solution of degree at most $r$ ($r\ge 1$). So, we introduce a family of regular triangulation (made of simplexes)  $\mathcal{T}_h$ of the domain $\overline{\Omega}$ and set $h=\max_{\rho\in \mathcal{T}_h}(h_\rho)$, where $h_{\rho}$ denotes the diameter of the mesh element $\rho$. Let $V_h \subset H^1_0(\Omega)$ denote the usual conforming finite element space of continuous, piecewise polynomial functions of degree at most $r$  on  
$\mathcal{T}_h$ that vanish on $\partial \Omega$. Let $\vV_h=[V_h]^d$ be the finite element vector space. We define the parametric finite element approximate solution  as: find $\tilde \vu_h(\cdot, \bsy_s) \in \vV_h$ such that 
\begin{equation}\label{FE solution}
 \calB(\bsy_s;\tilde \vu_h, \vv_h) = \ell(\vv_h), \quad \text{for all} \quad \vv_h \in \vV_h,\quad{\rm for~every}~~\bsy_s \in U.
\end{equation}
In the next theorem, we  discuss the error estimate from the finite element discretization.    For the proof, we refer to \cite[Theorem 6]{ClarkeEtAl2024}. We assume that the truncated continuous solution $\tilde \vu$ satisfies the following regularity property: for some $1\le \theta\le r,$
\begin{equation}\label{a priori H2}
  \|\tilde \vu(\cdot,\bsy_s)\|_{{\bf H}^{\theta+1}}
 \lesssim\,\|\vf\|_{{\bf H}^{\theta-1}},\quad{\rm  for~ every}~~ \bsy_s \in U \,.
\end{equation}
\begin{theorem}\label{Convergence theorem}
Assume that \eqref{ass A2} and \eqref{a priori H2} are satisfied. If $\chi:{\bf L}^2(\Omega) \to \R$ is a bounded  linear functional, 
then for every $\bsy_s  \in U,$ we have 
\[|\chi(\tilde \vu(\cdot, \bsy_s))-\chi(\tilde \vu_h(\cdot, \bsy_s))|\lesssim\, h^{\theta+1} \|\vf\|_{{\bf H}^{\theta-1}} \|\chi\|,\quad{\rm for}~~1\le \theta\le r.\]
The  constant $C$  depends on  $\Omega$, $\nu,$ $E_{\max}$ and $E_{\min}$, but not on $h$.
\end{theorem}
So far, we discussed the error from truncating the infinite series expansion in \eqref{eq:mulambda} and the error from approximating the truncated solution via the finite element method. The next task is to estimate the expected value of $\chi(\tilde \vu_h)$ using a higher-order QMC rule. 
\section{Higher order QMC method}\label{sec:QMC}
We approximate the $s$-dimensional integral $I_s(F)$ by an $N$-point QMC method $Q_{N,s}(F)$, which is an equal weight quadrature rule, i.e.
\begin{equation}\label{eq:IsF}
 I_s(F) \,:=\,
 \int_{[0,1]^s} F(\bsy) \,\rd\bsy \quad  \approx  \quad  Q_{N,s}(F) \; := \;
  \frac{1}{N} \sum_{n=0}^{N-1} F(\bsy_n)\;.
\end{equation}
The QMC points  $\{\bsy_0, \ldots,\bsy_{N-1} \}$ belong to $[0,1)^s$. We shall analyze, in particular, $Q_{N,s}$ being deterministic, {\emph interlaced high-order polynomial lattice} rules as introduced in \cite{Dick2008} and as considered for affine-parametric operator equations in \cite{DickKuoGiaNuynsSchwab2014}. To this end, to generate a polynomial lattice rule in base $b$ with $N$ points in $[0,1)^s$, we need a \emph{generating vector} of polynomials $\bg(x) = (g_1(x), \ldots, g_{s}(x)) \in [P_{m}({\mathbb Z}_{b})]^{s}$, where  $P_{m}({\mathbb Z}_{b})$ is the space of polynomials of degree less than $m$ in $x$ with coefficients taken from a finite field ${\mathbb Z}_{b}$ which, for convenience, we identify with the integers $\{0, 1, \ldots, b-1\}$.

For each integer $0\le n\le b^{m}-1$, we associate $n$ with the polynomial
$n(x) = \sum_{i=1}^{m} \eta_{i-1} x^{i-1}  \quad \in {\mathbb Z}_{b}[x]$,
where $(\eta_{m-1}, \ldots ,\eta_0)$ is the $b$-adic expansion of $n$, that is $n =\sum_{i=1}^{m} \eta_{i-1}\,b^{i-1}\,.$ We also need a map $v_{m}$ which maps elements in ${\mathbb Z}_{b}(x^{-1})$ to
the interval $[0,1)$, defined for any integer $w$ by
$v_{m} \left( \sum_{\ell=w}^\infty t_{\ell} x^{-\ell} \right) =\sum_{\ell=\max(1,w)}^{m} t_{\ell} b^{-\ell}.$

Let $P \in {\mathbb Z}_b[x]$ be an irreducible polynomial with degree $m$. The classical polynomial lattice rule $\calS_{P,b,m,s}(\bg)$ associated with $P$ and the generating vector $\bg$ is comprised of the quadrature points
\[\bsy_n =\left(v_{m} \Big( \frac{n(x)g_j(x)}{P(x)} \Big)\right)_{1\le j\le s} \in [0,1)^{s},\quad {\rm for}~~n = 0,\ldots, N - 1.\]

Classical polynomial lattice rules give almost first order of convergence for integrands of bounded variation. To obtain high-order of convergence, an interlacing procedure described as follows is needed. Following \cite{GodaDick2015}, the \emph{digit interlacing function}  with digit interlacing factor $\alpha \in \mathbb{N}$, $\mathscr{D}_\alpha: [0,1)^{\alpha}  \to  [0,1)$,   is defined  by
\[\mathscr{D}_\alpha(x_1,\ldots, x_{\alpha})= \sum_{i=1}^\infty \sum_{j=1}^\alpha
\xi_{j,i} b^{-j - (i-1) \alpha}\;,\]
where $x_j = \sum_{i\ge 1} \xi_{j,i}\, b^{-i}$ for $1 \le j \le \alpha$.
For vectors, we set $\mathscr{D}^s_\alpha: [0,1)^{\alpha s}  \to  [0,1)^s$ with 
$\mathscr{D}^s_\alpha(x_1,\ldots, x_{\alpha s}) =
(\mathscr{D}_\alpha(x_1,\ldots, x_\alpha),  \ldots,
\mathscr{D}_\alpha(x_{(s-1)\alpha +1},\ldots, x_{s \alpha})).$
Then, an interlaced polynomial lattice rule of
order $\alpha$ with $b^m$ points in $s$ dimensions
is a QMC rule using $\mathscr{D}_\alpha(\calS_{P,b,m,\alpha s}(\bg))$
as quadrature points, for some given modulus $P$ and generating vector $\bg$.

Next, we state an error bound for approximating the  integral ${\mathcal I}_{s}(F)$ by the QMC quadrature formula $Q_{s,N}(F)$ in \eqref{eq:IsF}. The proof relies on \cite[Theorem 3.1]{DickKuoGiaNuynsSchwab2014}.
\begin{theorem}\label{thm:quadrature}
Let $\bchi=(\chi_j)_{j\ge 1}$ be a sequence of positive numbers such that $\chi_j \ge \chi_{j+1}$ for all $j \in \mathbb{N}$ and $\sum_{j=1}^\infty \chi_j^p$ being finite for some $0<p<1$. Let $ \bchi_{s}=(\chi_j)_{1\le j\le s}$ and let $\alpha \,:=\, \lfloor 1/p \rfloor +1$. Assume that $F$ satisfies the following regularity properties: for any $\bsy\in [0,1)^s$,  
$\balpha \in \{0, 1, \ldots, \alpha\}^{s}$,  the following inequalities hold $|\partial^{\balpha}_\bsy {F}(\bsy)| \le c|\balpha|! 
\bchi_{s}^{\balpha},$ where the constant $c$ is independent of $\bsy$, $s$ and of $p$. Then one can construct an interlaced polynomial lattice rule of order $\alpha$ with $N$ points, using a fast component-by-component (CBC) algorithm,  with cost $\calO(\alpha\,s N(\log N+\alpha\,s))$   operations, respectively,  so that the following error bound holds
\[|{\mathcal I}_{s} (F)-Q_{s,N} (F)|\le C\,N^{-1/p}\,.\]
The  constant $C$ depends on $b$ and $p$, but is independent of $s$ and $m$. 
\end{theorem}
\section{Bayesian inversion}\label{Bayesian}
We consider the following {\em inverse problem}:
given observational data $\delta$,
predict a ``most likely'' value of a
Quantity of Interest (QoI)
$\phi$, which is typically a continuously differentiable functional of the uncertain Young's modulus $E$. Let $\calG$ be the solution operator of equation \eqref{eq:L1}. That is, $\calG(E) = \vu.$

We assume given an observation functional $\calO(\cdot): \vV \rightarrow Y$ (these are the bounded linear functionals $\chi$ in the forward problem), which denotes a
{\em bounded linear observation operator}
on the space $\vV$ of observed system responses in $Y$.
We assume that there is a finite number $K$ of
sensors, and each sensor can measure the values of $2$ components of the data, so that $Y = \mathbb{R}^{2K}$.
Then $\calO\in \calL(\vV; Y) \simeq (\vV^*)^{2K}$.
We equip $Y$ with the Euclidean norm, denoted by $\|\cdot\|$.
For example, if $\calO(\cdot)$ is a
$K$-vector of observation functionals
$\calO(\cdot) = (o_k(\cdot))_{k=1}^{2K}$.
Assume the observation noise covariance is $\Gamma$, the covariance-weighted, least squares potential $\Phi_{\Gamma}: \vV \times Y \to \mathbb{R}$ is given by
\begin{equation}
\Phi_\Gamma(E;\delta) = \frac{1}{2} \|\delta - \calO(\calG(E))  \|^2_{\Gamma}
:= \frac{1}{2}
((\delta - \calO(\calG(E)))^\top \Gamma^{-1}
(\delta - \calO(\calG(E)))).
\end{equation}
The following is a variant of \cite[Theorem 3.4]{Dashti-Stuart2017}
\begin{theorem}\label{thm:Bayes}
Assume that the potential
$\Phi_\Gamma(E;\delta) = \frac{1}{2} \|\delta - \calO(\calG(E)) \|^2_{\Gamma}$ is,
for fixed data $\delta \in \R^{2K}$,
$\bpi_0$ measurable and that, for $\mathbb{Q}_0$-a.e.~data $\delta\in \R^{2K}$ there holds
$$
Z:= \int_U\! \exp\left( -\Phi_{\Gamma}(E;\delta) \right) \bpi_0(\dd u) > 0 \;.
$$
Then the conditional distribution of $\vu|\delta$ ($\vu$ given $\delta$) exists and is denoted by $\bpi^\delta$.
It is absolutely continuous with respect to
$\bpi_0$ and there holds
\[\frac{d\bpi^\delta}{d\bpi_0}(E) 
=
\frac{1}{Z}  \exp\left( -\Phi_{\Gamma}(E;\delta) \right) 
\;.\]
\end{theorem}
Let
\[
\Theta(\bsy) = 
\exp\bigl(-\Phi_\Gamma(E;\delta)\bigr),
\quad
\Theta(\bsy_s) = 
\exp\bigl(-\Phi_\Gamma(E_s;\delta)\bigr).
\]
With the quantity of interest (QoI) $\phi:{\bf L}^2(\Omega) \to \R$, we associate the parametric map
\[\Psi(\bsy)
=
\Theta(\bsy)\phi(\vu(E))
=
\exp\bigl(-\Phi_\Gamma(E;\delta)\bigr)\phi(\vu(E)) 
\;.\]
The Bayesian estimate of the QoI $\phi$, given noisy data $\delta$, takes the form
\[\mathbb{E}^{\bpi^\delta}[\phi]
= 
Z^\prime/Z, \;\;
\quad 
Z^\prime:= I(\Psi) = 
\int_{U} \! \Psi(\bsy) \,\bpi_0(\dd\bsy)
\;.\]

In the following, we propose an approximation strategy to compute the high dimensional integrals $Z, Z'$ using a higher-order QMC method. 
The truncated parametric map applied to the truncated forward problem is defined by
\[
\Psi(\bsy_s) = \Theta(\bsy_s) \phi(\tilde \vu(E_s)).
\]
Thus,  the truncated parametric map applied to the finite element approximation to the forward problem is defined by
\[
\Psi_h(\bsy_s) = \Theta(\bsy_s) \phi(\tilde \vu_h(E_s)).
\]

\begin{theorem}\label{thm-Zprime est}
Assume that  \eqref{ass A11}--\eqref{ass A2} and \eqref{a priori H2} are satisfied. In addition, the QoI mapping $\phi$ is assumed to be a bounded linear functional. Let $Z^\prime_{N,s,h} = Q_{N,s}(\Psi_h)$, and $Z^\prime =I(\Psi)$,
then we have the following error bound
\[
|Z^{\prime}_{N,s,h} - Z^\prime|
\le C (h^{\theta+1} + N^{-1/p} + 
s^{1-1/p}),\quad{\rm for}~~1\le \theta\le r.
\]    
\end{theorem}
\begin{proof}
Let $Z^\prime_{N,s} = Q_{N,s}(\Psi)$ and $Z'_{s}=I_s(\Psi)$.
Then, by using the triangle inequality, we have
\[
|Z^{\prime}_{N,s,h} - Z^\prime|
\le |Z^{\prime}_{N,s,h}-Z^\prime_{N,s}|+|Z^{\prime}_{N,s} - Z^\prime_s| + |Z^\prime_s - Z^\prime|.
\]
We now estimate term by term. An application of Theorem~\ref{Convergence theorem} yields 
\begin{align*}
|Z^\prime_{N,s,h} - Z^\prime_{N,s}|
  &= | Q_{N,s}(\Psi_h) - Q_{N,s}(\Psi)|
 \le \|\Psi_h - \Psi\|_\infty
 \\
 &\le \|\Theta\|_\infty
 |\phi(\tilde \vu)  - \phi(\tilde \vu_h)|
 \le Ch^{\theta+1}\|\Theta\|_\infty
 \|\phi\| \|\vf\|_{{\bf H}^{\theta-1}} 
\end{align*}
To estimate the quadrature error, we use  Theorem~\ref{thm:quadrature} and obtain 
\begin{align*}
 |Z^\prime_{N,s} - Z^\prime_{s}|
 & = | Q_{N,s}(\Psi) - I_s(\Psi)| \le CN^{-1/p}.
\end{align*}
The truncation error can be estimated using Theorem~\ref{convergence calL u-us} as 
\[|Z^\prime_{s}   - Z^\prime|=
  |I_s(\Psi) - I(\Psi)|
  \le C \max_{\bsy \in U} |\Theta(\bsy)| |\phi(\vu)-\phi(\tilde \vu)|
  \le C s^{1-1/p}\,.\]
Combining all error estimates, we arrive at the result of the theorem.
\end{proof}

In Theorem~\ref{thm-Zprime est} we assume that the normalizing constant $Z$ is known exactly, or at least, can be approximated sufficiently accurately as to not influence the error. For the case where the error of approximating $Z$ is also included in the error analysis see \cite{DickGantnerLeGiaSchwab2019}.

\section{Numerical experiments}\label{Numerical}
We choose $\Omega = [0,1]^2$ to be a unit square. Assume that the Young's modulus depends on the parameters $\bsy = \{y_j\}_{j=1}^s$ via the finite expansion (so $\vu=\tilde \vu$)
\[
E (\bsx,\bsy) = 1 + 
\sum_{j=1}^s \frac{y_j}{j^2}
\sin(2\pi jx_1)\sin(2\pi(j+1)x_2), \quad \bsx = (x_1,x_2) \in \Omega.
\]
As prior measure $\pi_0$ on $(-1/2,1/2)^s$, we use the uniform product measure $\pi_0(\dd u) = \otimes_{j=1}^s  \dd y_j$.

We set the Poisson ratio $\nu = 0.4$ and $\vf (\bsx) = (2x_1 + 10,~x_2 - 3).$ The forward map $\calG: W^{\theta,\infty}(\Omega) \rightarrow \vV$, defined by $\calG(E) = \vu$, is approximated by the quadratic finite element solver $\calG_h(E) = \tilde \vu_h$. We choose $K=10$ observation points with Cartesian coordinates; $\bsx_k = 
(0.5, 10^{-3}+ k(10^{-1}-10^{-4}))$ for $k=0,1,\ldots,9.$ The noisy observations are point evaluations of the solution
\[
o_k = \vu(\bsx_k, \bsy^*) + \eta_k, \quad k=0,\ldots,9.
\]
with $\eta_k \sim \calN(0,\Gamma)$,
where $\calN(0,\Gamma)$ denote the multivariate normal distribution with
covariance matrix $\Gamma$. In the experiments, we choose $\Gamma = \sigma {\bf I}$, with $\sigma=0.1$ and ${\bf I}$ being the $2K\times 2K$ identity matrix.
To generate the observations, the truth $\bsy^*$ was chosen as a sample in the prior. In Figure~\ref{fig:post-density-s-2}, we set $s=2$ and plot the un-normalized posterior density, that is, the function $\exp(-\Phi_{\Gamma}(E;\delta))$ over the uniform grid $(-1/2,1/2)^{2}$. The linear functional $\phi$ is defined by
\[
\phi(\vu) = \int_{\Omega} (u_1 + u_2) d\bsx. 
\]
We then set $s=64$, generated another set of measurements with noise and computed the approximate integrals $Z'_N/Z$ with different values of $N$. We used higher-order QMC point sets based on those presented in \cite{GS14_575}.
Table \ref{tab:ErrorQn-s-64} shows the errors of $Z'_N/Z$ compared to the reference value (computed with $N=2^{15}=32768$ quadrature points) 
 with numerical expected order of convergence (EOC). 
\begin{figure}
    \centering
    \includegraphics[width=0.8\linewidth]{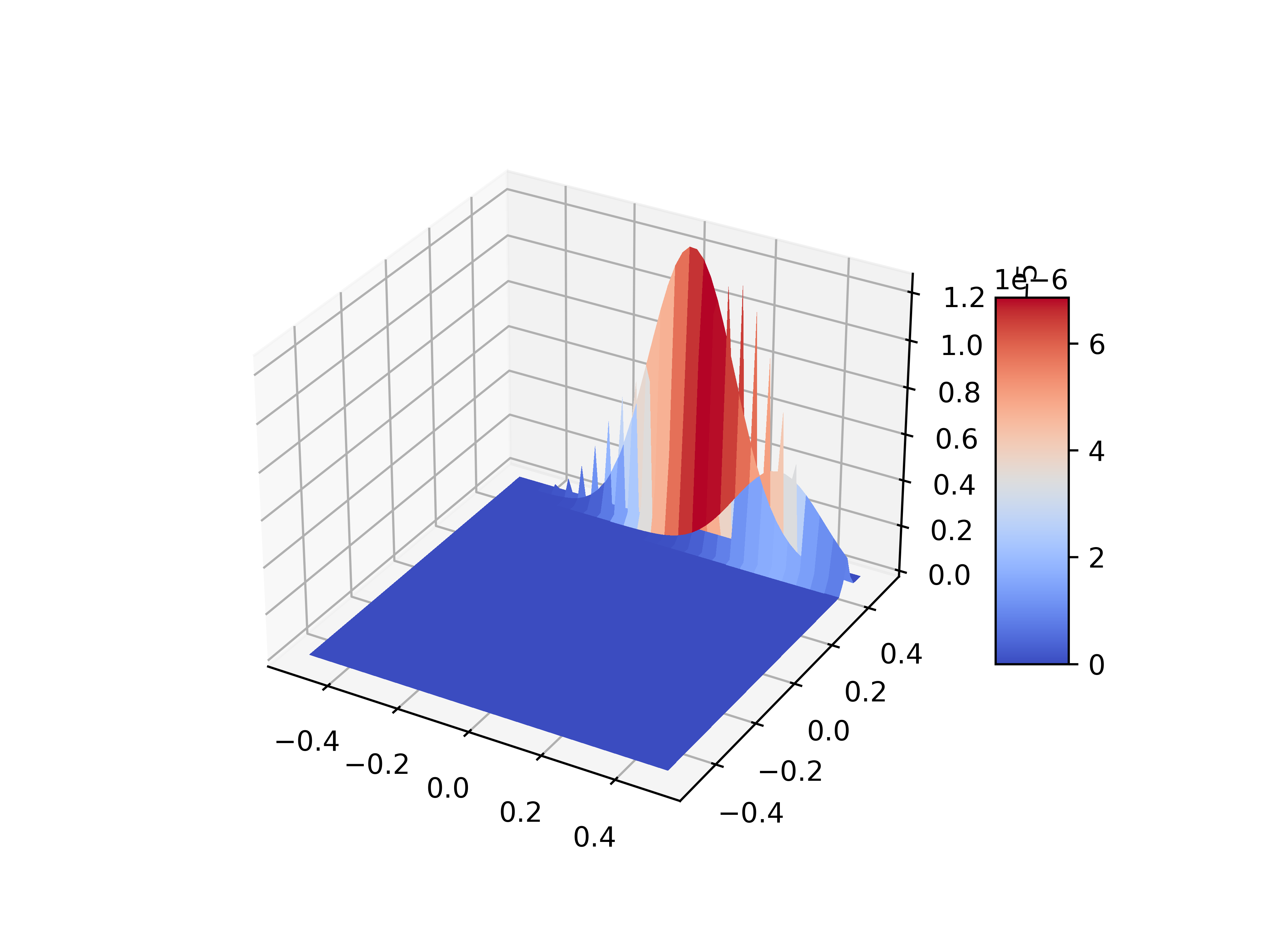}
    \caption{Un-normalised posterior density with $s=2$}
    \label{fig:post-density-s-2}
\end{figure}

\begin{table}[]
\centering
\begin{tabular}{|c|c|c|}\hline
$N$ & $\text{err}_N$  & EOC \\
\hline        
256  & 0.8758 & \\
1024 & 0.3500 &  0.6617 \\
2048 & 0.1193 &  1.5527 \\
4096 & 0.0176 &  2.7612 \\
\hline
 \end{tabular}
    \caption{Errors of $Z^{'}_N/Z$ compared to the reference value.}
    \label{tab:ErrorQn-s-64}
\end{table}
\paragraph{Acknowledgements}
The authors gratefully acknowledge the support of the Australian Research Council, in particular, ARC Discovery Grant DP220101811.
\ifx\printbibliography\undefined
    \bibliographystyle{plain}
    \bibliography{ctac}
\else\printbibliography\fi
\end{document}